\documentclass[11pt]{article}

\usepackage{authblk}
\usepackage{amsmath,amsthm,amssymb,enumitem,dsfont}
\usepackage{color}
\usepackage{hyperref}
\usepackage[margin=1.3in]{geometry}
\usepackage[title]{appendix}

\newcommand{\radius}{r}

\newcommand{\RR}{\mathbb{R}}

\newcommand{\ra}{\rightarrow}

\newcommand{\eg}{{\it e.g.}}
\newcommand{\ie}{{\it i.e.}}

\newtheorem{lem}{Lemma}
\newtheorem{thm}{Theorem}

\newtheorem{defi}{Definition}
\newtheorem{prop}{Proposition}
\newtheorem{obs}{Observation}

\begin{document}

\title{On Linear Optimization over Wasserstein Balls}

\author[1]{Man-Chung Yue}
\author[2]{Daniel Kuhn}
\author[3]{Wolfram Wiesemann}

\affil[1]{\small \textit{Department of Applied Mathematics, The Hong Kong Polytechnic University,} \texttt{manchung.yue@polyu.edu.hk}}
\affil[2]{\small \textit{College of Management of Technology, \'Ecole polytechnique f\'ed\'erale de Lausanne,} \texttt{daniel.kuhn@epfl.ch}}
\affil[3]{\small \textit{Imperial College Business School, Imperial College London,} \texttt{ww@imperial.ac.uk}}

\date{\today}

\maketitle

\begin{abstract}
Wasserstein balls, which contain all probability measures within a pre-specified Wasserstein distance to a reference measure, have recently enjoyed wide popularity in the distributionally robust optimization and machine learning communities to formulate and solve data-driven optimization problems with rigorous statistical guarantees. In this technical note we prove that the Wasserstein ball is weakly compact under mild conditions, and we offer necessary and sufficient conditions for the existence of optimal solutions. We also characterize the sparsity of solutions if the Wasserstein ball is centred at a discrete reference measure. In comparison with the existing literature, which has proved similar results under different conditions, our proofs are self-contained and shorter, yet mathematically rigorous, and our necessary and sufficient conditions for the existence of optimal solutions are easily verifiable in practice.
\end{abstract}

\section{Introduction}\label{sec:intro}

Let $(X, d)$ be a Polish (\ie, a complete and separable) metric space. We assume that $(X, d)$ is proper, that is, for any $R>0$ and $x_0\in X$, the closed ball $B_R(x_0) := \{x\in X: d(x,x_0) \le R\}$ is compact. Examples of proper metric spaces include finite-dimensional Banach spaces and complete Riemannian manifolds. We denote by $\mathcal{P} (X)$ the set of all Borel probability measures that are supported on $X$.

For $p \in [1, \infty)$, the $p$-th Wasserstein distance between $\mu_1, \mu_2 \in \mathcal{P} (X)$ is defined as 
\begin{equation}\label{def:W_P}
	W_p (\mu_1 , \mu_2) = \inf_{\gamma \in \Gamma(\mu_1, \mu_2)}  \left(  \int_{X\times X} d^p (x_1,x_2)\,\mathrm{d}\gamma(x_1, x_2) \right)^{\frac{1}{p}} ,
\end{equation}
where $\Gamma (\mu_1, \mu_2)$ is the set of all couplings of $\mu_1$ and $\mu_2$, that is, the set of all probability measures supported on $X \times X$ with marginals $\mu_1$ and $\mu_2$ (see, \eg, \cite[Definition 2.1]{clement2008wasserstein} or \cite[Definition 6.1]{villani2008optimal}). Intuitively speaking, $W_p (\mu_1 , \mu_2)$ measures the minimum transportation cost required to transform the mass of $\mu_1$ into the mass of $\mu_2$, where the transportation cost is measured according to the ground metric $d$.

The Wasserstein ball of radius $\radius > 0$ centred at the reference measure $\nu \in \mathcal{P} (X)$ is
\begin{equation}\label{def:the_ball}
	\mathcal{B}_\radius(\nu) = \left\lbrace \mu \in \mathcal{P}(X) : W_p(\mu, \nu) \le \radius \right\rbrace.
\end{equation}
In this paper, we study the optimization problem
\begin{equation}\label{opt:dist}
	\begin{array}{l@{\quad}l@{\quad}l}
		\displaystyle \mathop{\text{maximize}}_{\mu} & \displaystyle \int_X f(x) \, \mathrm{d}\mu (x) \\[4mm]
		\displaystyle \text{subject to} & \displaystyle \mu \in \mathcal{B}_{\radius} (\nu),
	\end{array}
\end{equation}
where $f : X \ra \RR$ is assumed to be upper semi-continuous, and where for each $\mu\in \mathcal{B}_\radius (\nu)$ at least one of the integrals $\int_X [f (x)]_+ \, \mathrm{d} \mu (x)$ and $\int_X [-f (x)]_+ \, \mathrm{d} \mu (x)$ is finite. Optimization problems of the form~\eqref{opt:dist} arise in distributionally robust optimization, where ambiguity averse decisions are sought that perform well under misestimations of the unknown true data generating distribution \cite{BM19:quantifying, gao2016distributionally, MEK18:watergate, ref:nguyen2019calculating, PW07:ambiguity, ZG18:watergate}. Problem~\eqref{opt:dist} also emerges when one regularizes machine learning problems such as classification problems \cite{NIPS2015_5679}, clustering problems \cite{10.5555/3305381.3305536} or generative adversarial networks \cite{wgan2017} against overfitting to the training samples.

The remainder of this technical note proceeds as follows. We first prove in Section~\ref{sec:weakly-compact} that the Wasserstein ball $\mathcal{B}_\radius(\nu)$ is weakly compact under mild conditions. We then leverage this finding in Section~\ref{sec:finiteness} to derive a necessary and sufficient condition for the optimal value of the optimization problem~\eqref{opt:dist} to be finite. Section~\ref{sec:existence_of_solutions} shows that essentially the same condition is also sufficient for the optimal value of~\eqref{opt:dist} to be attained. Section~\ref{sec:discrete} is devoted to sparse solutions to~\eqref{opt:dist} that place positive probability on finitely many atoms. To keep this note self-contained, we review results from measure theory and from infinite-dimensional linear programming in the Appendices~\ref{sec:m-theory} and~\ref{sec:sparsity}, respectively.

\section{Weak Compactness of the Wasserstein Ball}\label{sec:weakly-compact}

We say that a probability measure $\mu$ has a finite $p$-th moment if $\int_X d^p(x,x_0) \, \mathrm{d}\mu(x) < \infty$ for some $x_0 \in X$. The triangle inequality for $d$ implies that the integral is finite for some $x_0 \in X$ if and only if it is finite for all $x_0 \in X$, and thus the reference point $x_0$ does not matter. In contrast to the literature, which commonly defines the Wasserstein distance for measures with finite $p$-th moment only, our definition of the Wasserstein distance also applies to measures that do not possess a finite $p$-th moment. We now show, however, that the Wasserstein ball $\mathcal{B}_\radius (\nu)$ only contains measures with a finite $p$-th moment whenever the reference distribution $\nu$ has a finite $p$-th moment.

\begin{lem}\label{lem:main}
The Wasserstein ball $\mathcal{B}_\radius (\nu)$ has a uniformly bounded $p$-th moment, that is, there exists $C > 0$ such that $\int_X d^p(x,x_0) \, \mathrm{d}\mu(x) \leq C$ for all $\mu \in \mathcal{B}_\radius (\nu)$, whenever $\nu$ has a finite $p$-th moment.
\end{lem}

\begin{proof}[Proof of Lemma~\ref{lem:main}]
Since the reference measure $\nu$ has a finite $p$-th moment, there is $x_0 \in X$ and $C_0 \in \mathbb{R}$ such that $\int_X d^p(x,x_0) \, \mathrm{d}\nu(x) \leq C_0$. We claim that $\int_X d^p(x,x_0) \, \mathrm{d}\mu(x) \leq C_0 + r$ for all measures $\mu \in \mathcal{B}_\radius (\nu)$ in the Wasserstein ball. To see this, note that
\begin{equation*}
	\begin{split}
		&\,\left( \int_X d^p(x,x_0)\, \mathrm{d}\mu (x) \right)^{\frac{1}{p}} 
		= W_p (\mu, \delta_{x_0}) 
		\le W_p (\nu, \delta_{x_0}) + W_p (\mu, \nu) \\
		=&\, \left( \int_X d^p (x, x_0)\, \mathrm{d}\nu (x) \right)^{\frac{1}{p}}+ W_p (\mu, \nu) 
		\le C_0 + \radius,
	\end{split}
\end{equation*} 
where $\delta_{x_0}$ is the Dirac measure at $x_0$, and the first inequality follows from the triangle inequality of the Wasserstein distance on $\mathcal{P} (X)$~\cite[Corollary 5.7]{clement2008wasserstein}.
\end{proof}

We are now ready to prove the weak compactness of $\mathcal{B}_\radius (\nu)$.

\begin{thm}[Weak Compactness of Wasserstein Ball]\label{thm:compact}
The Wasserstein ball $\mathcal{B}_\radius(\nu)$ is weakly compact whenever the reference measure $\nu$ has a finite $p$-th moment.
\end{thm}

\begin{proof}[Proof of Theorem~\ref{thm:compact}]
The map $\mu \mapsto W_p (\mu, \nu)$ is lower semi-continuous with respect to the weak convergence \cite[Corollary 5.3]{clement2008wasserstein}. Since the Wasserstein ball $\mathcal{B}_\radius (\nu)$ is a lower level  set of this map, $\mathcal{B}_\radius (\nu)$ is weakly closed, that is, $\mathcal{B}_\radius (\nu)$ coincides with its weak closure. To see that $\mathcal{B}_\radius (\nu)$ is weakly compact, we show that $\mathcal{B}_\radius (\nu)$ is tight. We can then employ Prokhorov's Theorem (\emph{cf.}~Theorem~\ref{thm:Prokhorov} from Appendix~\ref{sec:m-theory} and the subsequent remark) to conclude that the closure of the Wasserstein ball $\mathcal{B}_\radius (\nu)$, which by our previous argument coincides with $\mathcal{B}_\radius (\nu)$, is weakly compact.

By Lemma~\ref{lem:main}, there is $C > 0$ such that $\left[ \int_X d^p(x,x_0) \, \mathrm{d}\mu(x) \right]^{\frac{1}{p}} \leq C$ for all $\mu \in \mathcal{B}_\radius (\nu)$. To see that the Wasserstein ball $\mathcal{B}_\radius (\nu)$ is tight, we show that for every $\epsilon > 0$, we have $\mu \left( X \setminus B_{C/\epsilon} (x_0) \right) \leq \epsilon$ for all $\mu \in \mathcal{B}_\radius (\nu)$ (\emph{cf.}~Definition~\ref{def:tight} from Appendix~\ref{sec:m-theory}). Indeed, we have that
\begin{align*}
\mu \left( X \setminus B_{C/\epsilon} (x_0) \right)
\;\; &= \;\;
\mu(\lbrace x \in X \, : \, d(x,x_0) > C/\epsilon \rbrace)
\;\; \le \;\;
\frac{\int_X d(x,x_0) \, \mathrm{d}\mu(x)}{C/\epsilon} \\
\;\; &\le \;\;
\frac{\epsilon \cdot \left[ \int_X d^p (x,x_0) \, \mathrm{d}\mu(x) \right]^{\frac{1}{p}}}{ C }
\mspace{45mu} \le \;\;
\epsilon,
\end{align*}
where the first inequality is due to Markov's inequality, the second inequality is due to Jensen's inequality, and the third inequality holds by our definition of $C$. Note that the set $B_{C/\epsilon} (x_0)$ is compact because of the properness of $(X,d)$.
\end{proof}

Statements similar to Theorem~\ref{thm:compact} have been shown in \cite[Lemma~3.34]{DelloSchiavo2015} and \cite[Proposition~3]{pichler2017quantitative}. We conclude this section by showing that our assumption of the Polish space $(X, d)$ being proper is indeed necessary for the statement of Theorem~\ref{thm:compact}.\footnote{We are grateful to Lorenzo Dello Schiavo, who communicated this result to us.}
\begin{obs}
\label{prop:proper}
	If the Wasserstein ball $\mathcal{B}_\radius(\nu)$ is weakly compact for every $\radius > 0$ and $\nu \in \mathcal{P}(X)$, then $(X,d)$ is proper.
\end{obs}
\begin{proof}[Proof of Observation~\ref{prop:proper}]
	We show that the closed ball $B_R(x_0) = \{x\in X: d(x,x_0) \le R\}$ is compact for any $R>0$ and $x_0\in X$. To this end, fix any $R>0$ and $x_0\in X$, and let $F: X \to \mathcal{P} (X)$ be defined as $F(x) = \delta_x$, where $\delta_x$ is again the Dirac measure at $x$. One readily  verifies that $F$ is weakly continuous. Together with the closedness of $B_R(x_0)$, this implies that the image $F( B_R(x_0) )$ is weakly closed. Since $F( B_R(x_0) ) \subseteq \mathcal{B}_R(\delta_{x_0})$ and $\mathcal{B}_R(\delta_{x_0})$ is weakly compact by assumption, we conclude that $F( B_R(x_0) )$ is weakly compact as well. Example~8.6.6 in \cite{bogachev2007measure2} then implies that $B_\radius (x_0)$ is compact.
\end{proof}

\section{Finiteness of the Optimal Value of Problem~\eqref{opt:dist}}\label{sec:finiteness}

We provide a necessary and sufficient condition for the finiteness of the optimal value of the optimization problem~\eqref{opt:dist}.

\begin{thm}\label{thm:finiteness}
	Assume that $\nu$ has a finite $p$-th moment. Then the optimal value of problem~\eqref{opt:dist} is finite if and only if there exist $x_0\in X$ and $c>0$ such that $f(x) \le c[1+ d^p (x,x_0)]$ for all $x\in X$.
\end{thm}

\begin{proof}
	Recall that the Wasserstein radius $r$ is strictly positive. By~\cite[Theorem 6.18]{villani2008optimal}, there thus exists a discrete probability measure $\hat{\nu} = \sum_{i = 1}^N \alpha_i \cdot \delta_{y_i}$, $\alpha \in \mathbb{R}^N_+$ with $\sum_{i = 1}^N \alpha_i = 1$, that is supported on $2 \leq N < \infty$ atoms $y_1,\dots, y_N \in X$ and that satisfies $W_p (\hat{\nu}, \nu) \le \radius / 2$, that is, $\hat{\nu}$ resides in the vicinity of $\nu$. We fix this measure for the remainder of the proof.
	
	Note also that $f$ is real-valued and the integral of $f$ under $\hat{\nu}$ reduces to a finite sum of real values. Since $\hat{\nu}$ is feasible in~\eqref{opt:dist}, the optimal value of~\eqref{opt:dist} thus cannot be $- \infty$. We next prove that the boundedness of~\eqref{opt:dist} implies the stated growth condition on $f$.
	
	Assume for the sake of contradiction that problem~\eqref{opt:dist} is bounded, but for all $x_0 \in X$ and $c > 0$ there is $x \in X$ for which $f (x) > c[1+ d^p (x,x_0)]$, in violation of the statement of the theorem. We can then construct a sequence $\{ y^k \}_k \subseteq X$ such that
	\begin{equation*}
	\lim_{k \ra \infty} \frac{f(y^k)}{1+d^p(y^k, y_1)} = +\infty.
	\end{equation*}
	Note that in this expression, we have chosen $x_0 = y_1$, where $y_1$ is the first atom of $\hat{\nu}$. In the following, we use the sequence $\{ y^k \}_k$ to construct a sequence of probability measures $\{ \mu^k \}_k$ from within the Wasserstein ball $\mathcal{B}_\radius (\nu)$ for which the integrals $\int_X f(x)\, \mathrm{d}\mu^k (x)$ diverge, in contradiction to the assumption that problem~\eqref{opt:dist} is bounded.
	
	Using the sequence $\{ y^k \}_k$, we construct the sequence of probability measures
	\begin{equation*}
	\mu^k = \epsilon^k \cdot \delta_{y^k} + (\alpha_1 - \epsilon^k) \cdot \delta_{ y_1} + \sum_{i =2}^N \alpha_i \cdot \delta_{ y_i}
	\quad \text{with} \quad
	\epsilon^k = \frac{\min\{\alpha_1, \radius^p/ 2^p\}}{1+d^p (y^k, y_1)}.
	\end{equation*}
	Note that each $\mu^k$ is indeed a probability measure since $\epsilon^k \in [0,\alpha_1]$ and $\epsilon^k + (\alpha_1 - \epsilon^k) + \alpha_2 + \ldots + \alpha_N = \alpha_1 + \ldots + \alpha_N = 1$. To see that $\mu^k \in \mathcal{B}_\radius (\nu)$ for all $k$, consider the transportation plan $\gamma^k \in \mathcal{P} (X\times X)$ given by
	\begin{equation*}
	\gamma^k = \epsilon^k \cdot \delta_{(y_1, y^k)} + (\alpha_1 - \epsilon^k) \cdot \delta_{( y_1,  y_1)} + \sum_{i = 2}^N \alpha_i \cdot \delta_{( y_i,  y_i)}.
	\end{equation*}
	One readily verifies that $\gamma^k \in \Gamma (\hat{\nu}, \mu^k)$, as well as
	\begin{align*}
	W_p (\nu, \mu^k)
	\;\; &\leq \;\;
	W_p (\nu, \hat{\nu}) + W (\mu^k, \hat{\nu})
	\;\; \leq \;\;
	\frac{\radius}{2} +
	\left[ \int_{X\times X} d^p (x_1,x_2) \,\mathrm{d}\gamma^k (x_1,x_2) \right]^{\frac{1}{p}} \\
	&= \;\;
	\frac{\radius}{2} +
	\left[ \epsilon^k \cdot d^p ( y_1, y^k)  + (\alpha_1 - \epsilon^k) \cdot d^p( y_1,  y_1) + \sum_{i = 2}^N \alpha_i \cdot d^p ( y_i,  y_i) \right]^{\frac{1}{p}}
	\;\; \leq \;\; r.
	\end{align*}
	Here, the first inequality follows from the triangle inequality of the Wasserstein distance on $\mathcal{P} (X)$~\cite[Corollary 5.7]{clement2008wasserstein}, the second inequality holds by  construction of $\hat{\nu}$ and because $\gamma^k \in \Gamma (\hat{\nu}, \mu^k)$, and the last inequality follows from the fact that $\epsilon^k \cdot d^p ( y_1, y^k) \leq r^p / 2^p$. We thus conclude that $\{ \mu^k \}_k \subseteq \mathcal{B}_\radius (\nu)$ as desired. To see that the integrals $\int_X f(x)\, \mathrm{d}\mu^k (x)$ diverge, we note that
	\begin{equation*}
	\mspace{-5mu}
	\lim_{k\ra \infty} \int_X f(x) \, \mathrm{d}\mu^k (x)
	\;\; = \;\;
	\lim_{k\ra \infty} \left[ \frac{\min\{\alpha_1, \radius^p/2^p\}}{1+d^p (y^k, y_1)} \cdot f(y^k) + (\alpha_1 - \epsilon^k) \cdot f( y_1) + \sum_{i=2}^N \alpha_i \cdot f( y_i) \right].
	\end{equation*}
	The second expression diverges since $f(y^k) / [1+d^p (y^k, y_1)] \longrightarrow \infty$ while $\min\{\alpha_1, \radius^p/2^p\}$ is constant in $k$. We have thus proved that the boundedness of problem~\eqref{opt:dist} implies the stated growth condition on $f$.
	
	To show that the stated growth condition on $f$ also implies the boundedness of problem~\eqref{opt:dist}, fix any $\mu \in \mathcal{B}_{\radius} (\nu)$ and note that
	\begin{align*}
	\int_X f(x) \, \mathrm{d}\mu (x)
	\;\; &\leq \;\; \int_X c \left[ 1+ d^p (x,x_0) \right] \, \mathrm{d}\mu (x)\\
	&= \;\; c + c \cdot W_p^p (\mu , \delta_{x_0}) \\
	&\leq \;\; c + c \left[ W_p (\mu , \nu) + W_p (\nu , \delta_{x_0})\right]^p,
	\end{align*}
	where the inequalities are due to the assumed growth condition on $f$ and the triangle inequality of $W_p$, respectively. Since $W_p (\mu , \nu) \leq r$ as $\mu \in \mathcal{B}_{\radius} (\nu)$, we can thus uniformly bound the objective value of every measure $\mu \in \mathcal{B}_{\radius} (\nu)$ in problem~\eqref{opt:dist}.
\end{proof}

\section{Existence of Optimal Solutions}\label{sec:existence_of_solutions}

We now study the existence of optimal solutions to the optimization problem~\eqref{opt:dist}. 

\begin{thm}\label{thm:optimiser_exists}
	Assume that $\nu$ has a finite $p$-th moment. 
	If there exist $x_0 \in X$, $ c > 0$ and $p' \in (0,p)$ such that $ f(x) \le c [1+ d^{p'} (x,x_0)]$ for all $x\in X$, then the optimal value of~\eqref{opt:dist} is attained.
\end{thm}

\begin{proof}
By Theorem~\ref{thm:compact} and Weierstrass' theorem, it suffices to show that the map $\mu \mapsto \int_X f (x) \,\mathrm{d}\mu (x) $ is weakly upper semi-continuous on $\mathcal{B}_\radius (\nu)$. To this end, let $\{\mu^k\}_k \subseteq \mathcal{B}_{\radius} (\nu)$ be a sequence  converging weakly to $\mu^\infty$. In order to show that $\limsup_{k\longrightarrow\infty} \int_X f(x) \, {\rm d}\mu^k(x)\leq \int_X f(x) \, {\rm d}\mu^\infty(x)$, we may may assume w.l.o.g.~that $\int_X f(x) \, \mathrm{d}\mu^k (x) > - \infty$ for all $k$. Since the assumptions of Theorem~\ref{thm:finiteness} are satisfied, the optimal value of problem~\eqref{opt:dist} is finite, and hence $\int_X f(x) \, \mathrm{d}\mu^k (x) < \infty$ for all $k$ as well. Since $\mathcal{B}_{\radius} (\nu)$ is weakly closed, we have $\mu^\infty \in \mathcal{B}_{\radius} (\nu)$. We need to show that for any $\epsilon > 0$ there is $k (\epsilon)$ such that
\begin{equation*}
\int_X f(x) \, \mathrm{d}\mu^k (x) - \int_X f(x) \, \mathrm{d}\mu^\infty (x) \le \epsilon
\qquad \forall k \geq k (\epsilon).
\end{equation*}

Our proof relies on the construction of an auxiliary function $f_R : X \rightarrow \mathbb{R}$ whose integrals under the probability measures $\mu \in \mathcal{B}_\radius (\nu)$ are close to those of $f$ but that is at the same time bounded from above by a constant involving $R$. To this end, fix $R > 0$ and set $f_R (x) = \min\{ f(x), \, c (R^{p'} + 1) \}$. We then have
\begin{equation}\label{eq:f_R}
	\left| f(x) - f_R (x) \right| \leq
	\begin{cases}
		c \cdot d^{p'}(x, x_0) & \text{if } d(x, x_0) \geq R, \\
		0 & \text{otherwise}
	\end{cases}
	\qquad \forall x \in X.
\end{equation}
To see this, note first that $f_R (x) \leq f (x)$ and thus $\left| f(x) - f_R (x) \right| = f(x) - f_R (x)$. Fix $x \in X$ and assume first that $d(x,x_0) \geq R$. If $f (x) \leq c(R^{p'} +1)$, then $f_R (x) = f (x)$ and thus $\left| f(x) - f_R (x) \right| = 0$. If $f (x) > c(R^{p'} +1)$, on the other hand, then 
\begin{equation*}
	f (x) - f_R (x)
 	\;\; = \;\;
 	f(x) - c(R^{p'}+1)
 	\;\; \leq \;\;
 	c(d^{p'}(x,x_0) +1) - c(R^{p'}+1)
 	\;\; \leq \;\;
 	c \cdot d^{p'}(x,x_0),
\end{equation*}
where first inequality follows from the assumption made in the statement of this theorem. Assume now that $d (x, x_0) < R$. In that case, the same assumption implies that $f(x) \leq c (d^{p'} (x,x_0) + 1) < c(R^{p'} +1)$ and hence $f(x) = f_R (x)$. We thus conclude that the bound~\eqref{eq:f_R} indeed holds. The bound implies that for all $\mu \in \mathcal{B}_\radius (\nu)$, we have that
\begin{align}
	\left| \int_{X} f(x) \, \mathrm{d}\mu (x) - \int_X f_R (x) \, \mathrm{d}\mu (x) \right|
	\;\; &\leq \;\;
	\int_{X} \left| f(x) - f_R (x) \right| \, \mathrm{d}\mu (x) \nonumber \\
	&\leq \;\;
	c \cdot \int_{ d(\,\cdot\,,x_0) \ge R }  d^{p'} (x,x_0)\,\mathrm{d}\mu (x)
	\;\; \leq \;\;
	\frac{c\cdot C}{R^{(p-p')}}, \label{eq:bound_for_f_R}
\end{align}
where the first and second inequality follow from the triangle inequality and equation~\eqref{eq:f_R}, respectively. The third inequality holds because
\begin{equation*}
	d^{p’} (x, x_0 )
	\;\; = \;\;
	\frac{d^p ( x, x_0 )}{d^{p – p’} (x, x_0)}
	\;\; \leq \;\;
	\frac{d^p (x, x_0)}{R^{(p – p’)}}
	\qquad \forall x \in X \, : \, d(x, x_0) \geq R,
\end{equation*}
and since Lemma~\ref{lem:main} implies that there is $C > 0$ such that
\begin{equation*}
	\int_{ d(\,\cdot\,,x_0) \ge R }  d^p (x,x_0)\,\mathrm{d}\mu (x)
	\;\; \leq \;\;
	\int_X d^p (x, x_0)\, \mathrm{d}\mu (x)
	\;\; \leq \;\;
	C
	\qquad\forall \mu \in \mathcal{B}_{\radius} (\nu).
\end{equation*}

Using the auxiliary function $f_R$, we can now prove the weak upper semicontinuity of the map $\mu \mapsto \int_X f (x) \,\mathrm{d}\mu (x)$ over $\mathcal{B}_\radius (\nu)$. Indeed, for every $\epsilon > 0$ there is $k (\epsilon)$ such that for all $k \geq k (\epsilon)$, we have
\begin{align*}
	\int_X f(x) \, \mathrm{d}\mu^k (x) - \int_X f(x) \, \mathrm{d}\mu^\infty (x)
	\;\; &\leq \;\;
	\int_X f_{R} (x) \, \mathrm{d}\mu^k (x) - \int_X f_{R} (x) \, \mathrm{d}\mu^\infty (x) \\
	& \quad
	+ \left| \int_X f(x) \, \mathrm{d}\mu^k (x) - \int_X f_{R} (x) \, \mathrm{d}\mu^k (x) \right| \\
	& \quad
	+ \left| \int_X f(x)\, \mathrm{d}\mu^\infty (x) - \int_X f_{R} (x) \, \mathrm{d}\mu^\infty (x) \right|
	\;\; \leq \;\;
	\epsilon.
\end{align*}
Here, the first inequality follows from the triangle inequality. As for the second inequality, we can choose $R$ sufficiently large such that $c \cdot C / [R^{(p-p')}] \leq \epsilon / 3$, which implies by equation~\eqref{eq:bound_for_f_R} that the two absolute values are both bounded above by $\epsilon / 3$. Note further that $f_R$ is upper semicontinuous and bounded from above by construction. Lemma~\ref{lem:Portmanteau} in Appendix~\ref{sec:m-theory} thus implies that the first difference of integrals can be made smaller than $\epsilon / 3$ by selecting $k (\epsilon)$ sufficiently large. We therefore conclude that $\mu \mapsto \int_X f (x) \,\mathrm{d}\mu (x)$ is weakly upper semicontinuous over $\mathcal{B}_\radius (\nu)$, as desired.
 \end{proof}

Note that the only difference between the conditions of Theorems~\ref{thm:finiteness} and~\ref{thm:optimiser_exists} is that in the latter case, we require the growth condition for $f$ to be satisfied for some $p' < p$. One can readily construct instances of problem~\eqref{opt:dist} where $f (x)$ grows asymptotically as $d_p^p (x, x_0)$ for some $x_0 \in X$ but where the optimal value is still attained. On the other hand, \cite[Example~4]{the_tutorial} presents an instance of~\eqref{opt:dist} that satisfies the condition of Theorem~\ref{thm:finiteness} and for which there exists no single measure that attains the (finite) optimal value. In this example, $X$ is necessarily unbounded. Indeed, if $X$ is bounded, then the Wasserstein distance $W_p$ metrizes the weak topology (see, \eg, \cite[page~330]{Dudley}, \cite[page~11]{Gibbs_Su} and \cite[Theorem 6.9]{villani2008optimal}) and therefore an optimal solution to problem~\eqref{opt:dist} exists even under the weaker condition of Theorem~\ref{thm:finiteness}.


A necessary and sufficient condition for the existence of optimal solutions to~\eqref{opt:dist} is derived in \cite{gao2016distributionally}. Since that condition requires the knowledge of an optimal dual solution to problem~\eqref{opt:dist}, however, it is difficult to verify in practice. In contrast, our result relies on a sufficient condition that is easily verifiable.

\section{Existence of Optimal Discrete Solutions}\label{sec:discrete}

We now show that we can restrict problem~\eqref{opt:dist} to probability measures supported on $N+1$ atoms if the reference measure $\nu$ of the Wasserstein ball $\mathcal{B}_\radius (\nu)$ is a discrete probability measure supported on $N$ atoms. This is frequently the case in applications, where the reference measure $\nu$ is chosen as the empirical measure on finitely many training samples.


\begin{thm}[Existence of Discrete Optimal Solutions]\label{thm:finite_supp}
Assume that $\nu$ is supported on $N$ atoms and that the conditions of Theorem~\ref{thm:optimiser_exists} hold. Then problem~\eqref{opt:dist} is optimized 
by a probability measure that is supported on at most $N + 1$ atoms.
\end{thm}

\begin{proof}
Our proof proceeds in several steps. By replacing the Wasserstein distance with its definition, we first transform the optimization problem~\eqref{opt:dist} into an infinite-dimensional linear program over the cone of nonnegative measures that acommodates $N$ equality constraints and $1$ inequality constraint. We next assume that the inequality constraint can be strengthened to an equality without affecting the optimal value of the problem. In this case, the problem becomes an infinite-dimensional linear program in standard form, and the desired conclusion follows from the sparsity of optimal basic feasible solutions to such problems (\emph{cf.}~Appendix~\ref{sec:sparsity}). If the inequality constraint cannot be strengthened to an equality without affecting the optimal value of the problem, finally, we show that the optimization problem admits an optimal solution that is a Dirac measure.

In view of the first step, note that the definition of the Wasserstein distance~\eqref{def:W_P} and the Wasserstein ball~\eqref{def:the_ball} imply that the optimization problem~\eqref{opt:dist} can be formulated as
\begin{equation*}
	\begin{array}{l@{\quad}l}
		\displaystyle \mathop{\text{maximize}}_{\mu} & \displaystyle \int_X f(x) \, \mathrm{d}\mu (x) \\[4mm]
		\displaystyle \text{subject to} & \mu \in \mathcal{P}(X) \\[1mm]
		& \displaystyle \inf_{\gamma \in \Gamma(\mu, \nu)} \int_{X \times X} d^p (x_1, x_2) \,\mathrm{d} \gamma (x_1, x_2) \le r^p .
	\end{array}
\end{equation*}
Following \cite{MEK18:watergate}, we can eliminate the embedded minimization over $\gamma$ in the problem above and obtain the equivalent single-level problem
\begin{equation}\label{opt:lift_0}
	\begin{array}{l@{\quad}l}
		\displaystyle \mathop{\text{maximize}}_{\mu, \, \gamma}  & \displaystyle \int_X f(x) \, \mathrm{d}\mu (x) \\[4mm]
		\mbox{subject to} & \mu\in \mathcal{P}(X), \;\; \gamma \in \mathcal{P} (X \times X) \\
		& \displaystyle P_1 \gamma = \mu, \;\; P_2 \gamma = \nu \\[1mm]
		& \displaystyle \int_{X \times X} d^p (x_1, x_2) \,\mathrm{d} \gamma (x_1, x_2) \le r^p,
	\end{array}
\end{equation}
where $P_1 \gamma$ and $P_2 \gamma$ refer to the first and second marginal probability measure of $\gamma$, respectively, that is, $(P_1 \gamma) (S) = \gamma (S \times X)$ and $(P_2 \gamma) (S) = \gamma (X \times S)$ for any Borel subset $S$ of $X$. Since $\mu = P_1 \gamma$, a change of variables 
allows us to rewrite the objective function as
\begin{equation*}
	\int_{X \times X} f(x_1) \, \mathrm{d}\gamma (x_1, x_2).
\end{equation*}

We can assume that $\nu$ satisfies $\nu = \sum_{i = 1}^N \alpha_i \cdot \delta_{y_i}$ with $\alpha \in \mathbb{R}^N_+$ and $\sum_{i=1}^N \alpha_i = 1$ as well as $y_1, \ldots, y_N \in X$. We can then re-express the constraint $P_2 \gamma = \nu$ in~\eqref{opt:lift_0} through
\begin{equation*}
	\int_{X\times X} \mathds{1}_{ X \times \{ y_i\} } (x_1, x_2) \,  \mathrm{d}\gamma (x_1, x_2) = \alpha_i
	\qquad \forall i = 1,\dots, N,
\end{equation*}
where the indicator function satisfies $\mathds{1}_S (x_1, x_2) = 1$ if $(x_1, x_2) \in S \subseteq X \times X$ and $\mathds{1}_S (x_1, x_2) = 0$ otherwise. If we additionally define $Y = \{ y_1, \dots,  y_N \}$ and make the normalization of $\gamma$ explicit, we obtain the following equivalent reformulation of problem~\eqref{opt:lift_0}:
\begin{equation}\label{opt:lift_0.5}
	\begin{array}{l@{\quad}l@{\qquad}l}
		\displaystyle \mathop{\text{maximize}}_{\gamma}  & \displaystyle \int_{X \times Y} f(x_1) \, \mathrm{d}\gamma (x_1, x_2) \\[4mm]
		\mbox{subject to} & \gamma \in \mathcal{M}_+ (X \times Y) \\[1mm]
		& \displaystyle \int_{X \times Y} \mathrm{d}\gamma (x_1, x_2) = 1 \\[4mm]
		& \displaystyle \int_{X \times Y} \mathds{1}_{ X \times \{ y_i \} } (x_1, x_2) \,  \mathrm{d}\gamma (x_1, x_2) = \alpha_i & \displaystyle \forall i = 1,\dots, N \\[4mm]
		& \displaystyle \int_{X \times Y} d^p (x_1, x_2) \,\mathrm{d} \gamma (x_1, x_2) \le r^p,
	\end{array}
\end{equation}
Here, $\mathcal{M}_+ (X \times Y)$ is the set of all non-negative finite Borel measures supported on $X \times Y$, and the first integral constraint ensures that $\gamma$ is indeed a probability measure residing in $\mathcal{P} (X \times Y)$. Note that any one of the $N + 1$ equality constraints in problem~\eqref{opt:lift_0.5} is implied by the remaining $N$ constraints. Hence, we can drop the normalization constraint. This seemingly redundant step ensures that we can get a discrete optimal solution with $N+1$ support points later. 
Combining the aforementioned reductions, problem~\eqref{opt:lift_0.5} becomes
\begin{equation}\label{opt:lift_1}
	\begin{array}{l@{\quad}l@{\qquad}l}
		\displaystyle \mathop{\text{maximize}}_{\gamma}  & \displaystyle \int_{X \times Y} f(x_1) \, \mathrm{d}\gamma (x_1, x_2) \\[4mm]
		\mbox{subject to} & \gamma \in \mathcal{M}_+ (X \times Y) \\[1mm]
		& \displaystyle \int_{X \times Y} \mathds{1}_{ X \times \{ y_i \} } (x_1, x_2) \,  \mathrm{d}\gamma (x_1, x_2) = \alpha_i & \displaystyle \forall i = 1,\dots, N \\[4mm]
		& \displaystyle \int_{X \times Y} d^p (x_1, x_2) \,\mathrm{d} \gamma (x_1, x_2) \le r^p.
	\end{array}
\end{equation}
Below, we distinguish between the two cases where the optimal value of problem~\eqref{opt:lift_1} changes (or remains unchanged) if we strengthen the last integral constraint to an equality.



We next assume that we can strengthen the last integral constraint to an equality without affecting the optimal value of the problem. In that case, problem~\eqref{opt:lift_1} resembles a linear program in standard form with $N + 1$ equality constraints and infinitely many nonnegative decision variables. We should thus expect that its optimal value is attained by a basic feasible solution, that is, a solution for which at most $N + 1$ variables are strictly positive and all others vanish. In our context, such a basic feasible solution would correspond to a discrete measure $\gamma$ that is supported on at most $N + 1$ atoms of $X \times Y$. To formalize this intuition, we apply Proposition~\ref{prop:moment} from Appendix~\ref{sec:sparsity} to conclude that our variant of problem~\eqref{opt:lift_1} is equivalent to
\begin{equation}\label{opt:lift_1.9}
	\begin{array}{l@{\quad}l@{\qquad}l}
		\displaystyle \mathop{\text{maximize}}_{\gamma}  & \displaystyle \int_{X \times Y} f(x_1) \, \mathrm{d}\gamma (x_1, x_2) \\[4mm]
		\mbox{subject to} & \gamma \in \mathcal{D}_{N+1} (X \times Y) \\[1mm]
		& \displaystyle \int_{X \times Y} \mathds{1}_{ X \times \{ y_i \} } (x_1, x_2) \,  \mathrm{d}\gamma (x_1, x_2) = \alpha_i & \displaystyle \forall i = 1,\dots, N \\[4mm]
		& \displaystyle \int_{X \times Y} d^p (x_1, x_2) \,\mathrm{d} \gamma (x_1, x_2) = r^p,
	\end{array}
\end{equation}
where we have replaced the set of all non-negative finite Borel measures $\mathcal{M}_+ (X \times Y)$ on $X \times Y$ with the subset of non-negative discrete measures $\mathcal{D}_{N+1} (X \times Y)$ that are supported on at most $N + 1$ points in $X \times Y$. Indeed, one readily verifies that our variant of problem~\eqref{opt:lift_1} is an instance of problem~\eqref{opt:infinite_LP} if we set $Z = X\times  Y$, $\psi = f \circ \pi_1$, where $\pi_1 (x_1, x_2) = x_1$ for $(x_1, x_2) \in X \times X$, $m = N+1$, $\phi_i = \mathds{1}_{X\times \{ y_i \}}$ and $v_i = \alpha_i$, $i = 1,\dots,N$, as well as $\phi_{N+1} = d^p$ and $v_{N+1} = r^p$. Moreover, the conditions of Proposition~\ref{prop:moment} are satisfied. Indeed, we assumed in Section~\ref{sec:intro} that at least one of the integrals $\int_X [f (x)]_+ \, \mathrm{d} \mu (x)$ and $\int_X [-f (x)]_+ \, \mathrm{d} \mu (x)$ is finite for each $\mu\in \mathcal{B}_\radius (\nu)$. Likewise, the requirement that $\int_Z |\phi_i| (z) \, \mathrm{d}\gamma (z) <\infty$ for all $i = 1,\dots,m$ is guaranteed by Lemma~\ref{lem:main} and the fact that $\gamma$ is a probability measure.


It is not a priori clear whether the feasible region of problem~\eqref{opt:lift_1.9} is weakly compact, and thus we cannot ascertain that the optimal value of~\eqref{opt:lift_1.9} is attained. If we replace the last equality in~\eqref{opt:lift_1.9} with an inequality, however, we obtain the relaxation
\begin{equation}\label{opt:lift_2}
\begin{array}{l@{\quad}l@{\qquad}l}
\displaystyle \mathop{\text{maximize}}_{\gamma}  & \displaystyle \int_{X \times Y} f(x_1) \, \mathrm{d}\gamma (x_1, x_2) \\[4mm]
\mbox{subject to} & \gamma \in \mathcal{D}_{N+1} (X \times Y) \\[1mm]
& \displaystyle \int_{X \times Y} \mathds{1}_{ X \times \{ y_i \} } (x_1, x_2) \,  \mathrm{d}\gamma (x_1, x_2) = \alpha_i & \displaystyle \forall i = 1,\dots, N \\[4mm]
& \displaystyle \int_{X \times Y} d^p (x_1, x_2) \,\mathrm{d} \gamma (x_1, x_2) \leq r^p
\end{array}
\end{equation}
whose optimal value is attained by Lemma~\ref{lem:optimal_value_attained} below. Note that the optimal value of~\eqref{opt:lift_2} is sandwiched by the optimal values of~\eqref{opt:lift_1}  and~\eqref{opt:lift_1.9}. Since the optimal values of~\eqref{opt:lift_1}  and~\eqref{opt:lift_1.9} coincide, we conclude that problem~\eqref{opt:lift_2} must have the same optimal value as well.

Let $\gamma^\star$ be an optimal solution to problem~\eqref{opt:lift_2}. Since $\gamma^\star \in \mathcal{D}_{N+1} (X \times Y)$, there are $(x_1^i, x_2^i) \in X \times Y$, $i = 1, \ldots, N+1$, as well as $\beta \in \mathbb{R}^{N+1}_+$ such that $\gamma^\star = \sum_{i=1}^{N+1} \beta_i \cdot \delta_{(x_1^i, x_2^i)}$ and $\sum_{i=1}^{N+1} \beta_i = 1$. By construction, $\mu^\star = P_1 \gamma^\star$ satisfies $\mu^\star \in \mathcal{D}^{N+1} (X)$, $\int_X f (x) \, \mathrm{d}\mu^\star (x) = \sup \eqref{opt:lift_2} = \sup \eqref{opt:dist}$, as well as
\begin{align*}
	W_p (\mu^\star, \nu)
	\;\; &= \;\;
	\inf_{\gamma \in \Gamma(\mu^\star, \nu)} \left(  \int_{X\times X} d^p (x_1,x_2)\,\mathrm{d}\gamma(x_1, x_2) \right)^{\frac{1}{p}} \\
	&\leq \;\;
	\left(  \int_{X\times X} d^p (x_1,x_2)\,\mathrm{d}\gamma^\star(x_1, x_2) \right)^{\frac{1}{p}}
	\;\; \leq \;\;
	r,
\end{align*}
which implies that $\mu^\star \in \mathcal{B}_r (\nu)$. We thus conclude that $\mu^\star$ is an $N+1$-point distribution in $\mathcal{B}_r (\nu)$ that optimizes problem~\eqref{opt:dist}, as desired.

Assume now that strengthening the last integral constraint of problem~\eqref{opt:lift_1} to an equality changes the optimal value of~\eqref{opt:lift_1}. We claim that in this case, any optimal solution $\mu^\star$ to problem~\eqref{opt:dist} satisfies the strict inequality $W_p (\mu^\star, \nu) < \radius$. Indeed, assume to the contrary that there is an optimal solution $\mu^\star$ to~\eqref{opt:dist} that satisfies $W_p (\mu^\star, \nu ) = r$. By \cite[Theorem 4.1]{villani2008optimal}, the distance $W_p (\mu^\star, \nu)$ is attained by some minimizer $\gamma^\star\in \mathcal{P} (X\times X)$ of~\eqref{def:W_P}, and one readily verifies that $\gamma^\star$ would constitute a feasible solution to problem~\eqref{opt:lift_1} that satisfies the last integral constraint as equality and that attains the optimal value of~\eqref{opt:lift_1}. This, however, contradicts our assumption that the optimal value of~\eqref{opt:lift_1} changes if we strengthen the last integral constraint to an equality.

We now claim that the Wasserstein ball $\mathcal{B}_r (\nu)$ must contain a Dirac measure that places all probability mass on a global maximizer of $f$. Assume to the contrary that $\mathcal{B}_r (\nu)$ does not contain such a Dirac measure. In that case, we must have $\arg \max_{x \in X} f(x) \neq \emptyset$, and any optimal solution $\mu^\star$ to~\eqref{opt:dist} must be supported on $\arg \max_{x \in X} f(x)$. Indeed, if that was not the case, there would be $\hat{x} \in X$ with $f (\hat{x}) > \int_X f(x) \, \mathrm{d}\mu^\star$. Consider now all convex combinations $\mu (\lambda) = \lambda \cdot \delta_{\hat{x}} + (1 - \lambda) \cdot \mu^\star$, $\lambda \in [0, 1]$. Since the map $\lambda \mapsto W_p^p (\mu (\lambda), \nu)$ is finite, convex and lower semi-continuous on $\lambda \in [0, 1]$, see \cite[Corollary 5.3]{clement2008wasserstein}, it is continuous on the entire interval. For sufficiently small $\lambda$, $\mu (\lambda)$ is therefore feasible in~\eqref{opt:dist} and attains a larger objective value than $\mu^\star$, thus violating the optimality of $\mu^\star$. Consider now any Dirac distribution $\delta_{x'}$ supported on $x' \in \arg \max_{x \in X} f(x)$. This Dirac distribution must be contained in $\mathcal{B}_r (\nu)$, for otherwise we could again form convex combinations $\mu'$ between $\mu^\star$ and $\delta_{x'}$ that are optimal in~\eqref{opt:dist} and that satisfy $W_p (\mu', \nu) = r$, in contradiction to our earlier finding. We thus conclude that the Wasserstein ball $\mathcal{B}_r (\nu)$ contains a Dirac measure that places all probability mass on a global maximizer of $f$, and this Dirac measure clearly constitutes an optimal discrete solution to~\eqref{opt:dist}.
\end{proof}

The following technical lemma is used in the proof of Theorem~\ref{thm:finite_supp}.

\begin{lem}\label{lem:optimal_value_attained}
	Assume that the assumptions of Theorem~\ref{thm:finite_supp} hold. If the optimal value of problem~\eqref{opt:lift_2} is finite, then it is attained.
\end{lem}

\begin{proof}
We first show that the feasible region of problem~\eqref{opt:lift_2} is weakly closed. Towards that end, note that the feasible region can be written as the intersection $ \mathcal{S}_1\cap \mathcal{S}_2\cap \mathcal{S}_3$ with
\begin{align*}
\mathcal{S}_1 & =\left\lbrace \gamma\in \mathcal{P}(X\times  Y): P_2 \gamma = \nu \right\rbrace = P_2^{-1} \left\lbrace\nu \right\rbrace, \\
\mathcal{S}_2 & = \left\lbrace \gamma\in \mathcal{P}(X\times  Y): \int_{X\times X} d^p (x,y) \, \mathrm{d}\gamma (x,y) \le \radius^p \right\rbrace \text{ and } \\
\mathcal{S}_3 & = \mathcal{D}_{N+1} (X\times  Y)\cap \mathcal{P}(X\times  Y),
\end{align*}
where $P_2$ sends a probability measure on $X\times Y$ to its second marginal distribution. We claim that these three sets are all weakly closed. Indeed, the weak closedness of $\mathcal{S}_1 = P_2^{-1} \left\lbrace\nu \right\rbrace$ follows from the facts that $P_2$ is continuous on $\mathcal{P}(X\times  Y)$ with respect to the weak topology by \cite[Theorem 15.14]{aliprantis2006infinite} and that any singleton in $\mathcal{P}(X\times  Y)$ is weakly closed. The weak closedness of $\mathcal{S}_2$ follows from the fact that it is a lower level set of the weakly lower semi-continuous map $\gamma \mapsto \int_{X\times X} d^p(x,y)\, \mathrm{d}\gamma(x,y)$, see~\cite[Lemma 5.3]{owhadi2017extreme}. As for the set $\mathcal{S}_3$, let $\Delta_{N+1}\subseteq \mathbb{R}^{N+1}$ be the probability simplex and $P:\Delta_{N+1}\times (X\times Y)^{N+1} \to \mathcal{S}_3$ be the map defined by
\begin{equation*}
P(\beta,z_1,\dots,z_{N+1}) = \sum_{i = 1}^{N+1} \beta_i \delta_{z_i},
\end{equation*}
which is weakly continuous. Noting that $\mathcal{S}_3  = P(\Delta_{N+1}\times (X\times  Y)^{N+1})$, we have $\mathcal{S}_3$ is also weakly closed. Hence, the feasible region of problem~\eqref{opt:lift_2} is also weakly closed.

Next, we claim that the feasible region is tight. Indeed, given any $\epsilon > 0$, by the tightness of the Wasserstein ball (see the proof of Theorem~1), there exists a compact set $K\subseteq X$ such that $\mu (X\setminus K) \le \epsilon$ for all $\mu \in \mathcal{B}_\radius(\nu)$. Also, for any $\gamma \in \mathcal{S}_1\cap \mathcal{S}_2$, we have that $P_1 \gamma \in \mathcal{B}_\radius (\nu)$. Therefore,
\begin{equation*}
\gamma \left( (X\times  Y)\setminus (K\times  Y) \right) = \gamma \left( (X \setminus K)\times  Y \right) = (P_1 \gamma) (X \setminus K) \le \epsilon,
\end{equation*}
which proves the claim. 
By the Prokhorov's Theorem (\emph{cf.}~Theorem~\ref{thm:Prokhorov} from Appendix~\ref{sec:m-theory} and the subsequent remark) and the two claims just proved, the feasible region is  weakly compact.

Finally, from the proof of Theorem~\ref{thm:optimiser_exists}, it can be easily proven that the objective function of problem~\eqref{opt:lift_2} \[\gamma\mapsto \int_{X \times Y} f(x_1) \, \mathrm{d}\gamma (x_1, x_2)\] is weakly upper semi-continuous on the feasible region. We thus conclude that an optimal solution exists. This completes the proof.
\end{proof}


One may wonder whether the result of Theorem~\ref{thm:finite_supp} can be strengthened further to the existence of optimal solutions to~\eqref{opt:dist} that are supported on fewer than $N + 1$ atoms. While this is possible for specific instances (for example, if $f (x)$ is concave), one can construct instances of problem~\eqref{opt:dist} where the optimal value is only attained by measures supported on at least $N + 1$ atoms \cite[Example~5]{the_tutorial}.

The sparsity of optimal solutions to problem~\eqref{opt:dist} has been investigated by several authors. To our best knowledge, the first result in this direction is \cite{wozabal2012framework}, which employs the Kantorovich-Rubinstein and the Richter-Rogosinski theorems to prove that if~\eqref{opt:dist} is solvable, then it is solved by a measure that is supported on at most $N + 3$ atoms. Subsequently, \cite{owhadi2017extreme} showed that, if problem~\eqref{opt:dist} is solvable, there are indeed optimal solutions that are only supported on at most $N + 2$ atoms. The sharp characterization of optimal measures supported on at most $N + 1$ atoms has been first derived in \cite{gao2016distributionally}. In contrast to our result, the authors do not employ the Richter-Rogosinski theorem or results remiscent of those in Appendix~\ref{sec:sparsity}. Instead, they rely on the first-order optimality conditions of the problem dual to~\eqref{opt:dist}. While this allows them to provide further insights into the structure of optimal solutions, their proof is substantially more difficult to verify than ours. Finally, we remark that a special case of this $(N+1)$-atom result has also been proved via yet another argument in~\cite{nguyen2019optimistic}.

\paragraph{Acknowledgements.}

The authors gratefully acknowledge funding from the Swiss National Science Foundation under Grant BSCGI0$\underline{~}$157733, the UK's Engineering and Physical Sciences Research Council under Grant EP/R045518/1 and the Hong Kong Research Grants Council under the Grant 25302420.

\newpage

\bibliographystyle{abbrv}
\bibliography{references}

\newpage

\begin{appendices}
	
\section{Auxiliary Measure-Theoretic Results}\label{sec:m-theory}

We review some well-known facts from measure theory that we use to prove our results. We first recall a connection between the notions of tightness and weak sequential compactness of collections of probability measures.

\begin{defi}\label{def:tight}
	A collection $\mathcal{S} \subseteq \mathcal{P} (X)$ of probability measures is tight if for any $\epsilon >0$, there exists a compact subset $B \subseteq X$ such that $\mu (X \setminus B) \le \epsilon$ for all $\mu\in \mathcal{S}$.
\end{defi}

\begin{defi}\label{def:weak_convergence}
	A sequence $\{\mu^k\}_k \subseteq \mathcal{P} (X)$ of probability measures converges weakly to $\mu^\infty \in \mathcal{P} (X)$ if for any bounded and continuous function $g$ on $X$, we have
	\begin{equation*}
		\lim_{k\longrightarrow \infty } \int_X g(x) \,\mathrm{d}\mu^k
		\;\; = \;\;
		\int_X g(x) \,\mathrm{d}\mu^\infty.
	\end{equation*}
\end{defi}

\begin{defi}\label{def:weak_compactness}
	A collection $\mathcal{S} \subseteq \mathcal{P} (X)$ of probability measures is weakly sequentially compact if every sequence in $\mathcal{S}$ \mbox{has a subsequence that converges weakly to an element of $\mathcal{S}$.}
\end{defi}

The concepts of tightness and weak sequential compactness are connected by Prokho-rov's Theorem, see for example \cite[Theorem 5.1]{billingsley2013convergence}.
	\begin{thm}[Prokhorov's Theorem]\label{thm:Prokhorov}
		A collection $\mathcal{S} \subseteq \mathcal{P} (X)$ of probability measures is tight if and only if the closure of $\mathcal{S}$ is weakly sequentially compact in $\mathcal{P}(X)$.
	\end{thm}
Note that the space $\mathcal{P}(X)$ is metrizable, sequential compactness and compactness of subsets of $\mathcal{P}(X)$ are equivalent to each other.	
	
	The following lemma, which is excerpted from the Portmanteau Theorem (see for example \cite[Problem 29.1(c)]{billingsley1995probability}), provides a useful characterization of weak convergence.
	\begin{lem}\label{lem:Portmanteau}
		A sequence $\{\mu^k \}_k \subseteq \mathcal{P} (X)$ of probability measures converges weakly to $\mu^\infty \in \mathcal{P} (X)$ if and only if for any upper bounded and upper semi-continuous function $g$ on $X$, we have
		\begin{equation*}
		\limsup_{k\longrightarrow \infty} \int_X g(x) \,\mathrm{d}\mu^k (x)
		\;\; \le \;\;
		\int_X g(x) \,\mathrm{d}\mu^\infty (x).
		\end{equation*}
	\end{lem}

	
	
\newpage

\section{Basic Feasible Solutions in Infinite-Dimensional Linear Programming}\label{sec:sparsity}

It is well-known that if a finite-dimensional linear program with $m$ equality constraints has an optimal solution, then there must be an optimal basic feasible solution with at most $m$ non-zero entries. An infinite-dimensional analogue of this fact is proved in \cite[Corollary 5 and Proposition 6(v)]{pinelis2016extreme}. To state this result, let $Z$ be a topological space, let $\mathcal{M}_+ (Z)$ be the set of non-negative finite Borel measures supported on $Z$, and let $\psi, \phi_1,\dots,\phi_m : Z \rightarrow \mathbb{R}$ be Borel functions as well as $v \in \mathbb{R}^m$. Consider now the optimization problem
\begin{equation}\label{opt:infinite_LP}
	\begin{array}{l@{\quad}l@{\quad}l}
		\displaystyle \mathop{\text{maximize}}_{\gamma} & \displaystyle \int_Z \psi (z) \, \mathrm{d}\gamma(z) \\[4mm]
		\displaystyle \text{subject to} & \gamma \in \mathcal{M}_+ (Z) \\[1mm]
		& \displaystyle \int_Z \phi_i (z) \, \mathrm{d}\gamma (z)= v_i & \displaystyle \forall i =1,\dots,m,
	\end{array}
\end{equation}
and denote by $\mathcal{F}$ the feasible region of~\eqref{opt:infinite_LP} and by $\mathrm{ext}(\mathcal{F})$ the set of extreme points of $\mathcal{F}$.

\begin{prop}\label{prop:moment}
	Suppose that for all $\gamma \in \mathcal{F}$, 	at least one of the integrals $\int_X [\psi (z)]_+ \, \mathrm{d} \gamma (z)$ and $\int_X [-\psi (z)]_+ \, \mathrm{d} \gamma (z)$ is finite and that $\int_Z |\phi_i| (z) \, \mathrm{d}\gamma (z) <\infty$ for all $i = 1,\dots,m$. If 
	\begin{equation}\label{eq:ext_cond}
		\sup\left\lbrace \int_Z \psi (z) \, \mathrm{d}\gamma (z): \gamma\in \mathcal{F} \right\rbrace \;\; = \;\;
		\sup\left\lbrace \int_Z \psi (z) \, \mathrm{d}\gamma (z): \gamma\in \mathrm{ext}(\mathcal{F}) \right\rbrace,
	\end{equation}
	then it holds that
	\begin{equation*}
		\sup\left\lbrace \int_Z \psi (z) \, \mathrm{d}\gamma(z) : \gamma\in \mathcal{F} \right\rbrace \;\; = \;\;
		\sup\left\lbrace \int_Z \psi (z) \, \mathrm{d}\gamma: \gamma\in \mathcal{F}\cap \mathcal{D}_m (Z)  \right\rbrace,
	\end{equation*}
	where $\mathcal{D}_m (Z)$ is the set of non-negative discrete measures supported on at most $m$ points in $Z$. Furthermore, if $\mathcal{F}\subseteq \mathcal{P}(Z)$ and $Z$ is Hausdorff, then the condition~\eqref{eq:ext_cond} is satisfied.
\end{prop}

We note that the conclusion of Proposition~\ref{prop:moment} cannot readily be drawn from the Richter-Rogosinski theorem \cite[Theorem 7.32]{Shapiro2009}. Indeed, in our context the Richter-Rogosinski theorem would only ensure the existence of  a non-negative discrete measure $\gamma^\star$ that is supported on at most $m + 1$ (instead of $m$) points since $\gamma^\star$ would have to satisfy $m + 1$ moment conditions: the $m$ moment constraints of problem~\eqref{opt:infinite_LP} as well as the additional constraint that $\gamma^\star$ attains the optimal objective value of problem~\eqref{opt:infinite_LP}.

\section{An Alternative Proof of Lemma~\ref{lem:optimal_value_attained}}

In this appendix, we provide an alternative proof of Lemma~\ref{lem:optimal_value_attained}, which is longer but arguably more elementary.

\begin{proof}
	Since problem~\eqref{opt:lift_2} has a finite optimal value, there is a sequence of feasible solutions $\{ \gamma^k \}_k$ to~\eqref{opt:lift_2} that attains the optimal value of~\eqref{opt:lift_2} asymptotically. We prove the lemma by constructing from $\{ \gamma^k \}_k$ a solution $\gamma^\star$ that is feasible in~\eqref{opt:lift_2} and that attains the optimal value of~\eqref{opt:lift_2}. To simplify the exposition, we assume in the representation of the reference measure $\nu$ that $\alpha_i > 0$ for all $i$ and that $y_i \neq y_j$ for all $i \neq j$; both conditions can always be satisfied by reducing the number of atoms $N$ if necessary.

	By going over to a subsequence if necessary, we may assume w.l.o.g.~that every measure $\gamma^k$ of the sequence $\{ \gamma^k \}_k$ can be represented as
	\begin{equation*}
		\gamma^k
		\;\; = \;\;
		\sum_{i = 1}^{N - 1} \alpha_i \cdot \delta_{(x_i^k, y_i)}
		\; + \;
		(\alpha_N - \beta_k) \cdot \delta_{(x_N^k, y_N)}
		\; + \;
		\beta_k \cdot \delta_{(x_{N+1}^k, y_N)}
	\end{equation*}
	for some $x_1^k, \ldots, x_{N+1}^k \in X$ and $\beta_k \in [0, \alpha_N]$. Indeed, every $\gamma^k$ is supported on $N + 1$ atoms from $X \times Y$ since $\gamma^k \in \mathcal{D}_{N+1} (X \times Y)$. Moreover, the first integral constraint in~\eqref{opt:lift_2} implies that $\gamma^k (X \times \{ y_i \}) = \alpha_i$ for all $i = 1, \ldots, N$. Since the reference measure $\nu$ has $N$ atoms but $\gamma^k$ is supported on $N + 1$ atoms, there is an atom $y_i$ of the reference measure $\nu$ whose probability mass $\alpha_i$ is split across two atoms in $\gamma^k$ (with one of them possibly having zero probability mass). While a different atom $y_i$ may be split for different $\gamma^k$, we can again go over to a subsequence if necessary to ensure that the same atom $y_i$ is split in all measures of $\{ \gamma^k \}_k$. Moreover, we can assume w.l.o.g.~that the split atom is $y_N$, which gives rise to the two atoms $(x_N^k, y_N)$ and $(x_{N+1}^k, y_N)$ for every $\gamma^k$; this can always be ensured by reordering the atoms of $\nu$ if necessary. In the remainder of this proof we argue that a subsequence of $\{ \gamma^k \}_k$ converges weakly to a measure $\gamma^\star$ that is feasible in~\eqref{opt:lift_2} and that achieves the optimal value of~\eqref{opt:lift_2}.
	
%
	
	We first show that $\{ \gamma^k \}_k$ has a subsequence for which $\{ x_i^k \}_k$ converges to some $x_i^\star \in X$ for every $i = 1, \ldots, N - 1$. Indeed, assume to the contrary that $\{ x_1^k \}_k$ does not have a convergent subsequence. Since the space $(X, d)$ is assumed to be proper, this is only possible if $\{ x_1^k \}_k$ diverges, that is, if $d (x_1^k, y_1) \longrightarrow \infty$. In that case, the associated transportation cost $\alpha_1 \cdot d (x_1^k, y_1)$ would also diverge since $\alpha_1 > 0$. Thus, the last constraint in problem~\eqref{opt:lift_2} would be violated, which contradicts the assumed feasibility of each member of the sequence $\{ \gamma^k \}_k$. Iteratively applying the same argument to $\{ x_i^k \}_k$, $i = 2, \ldots, N - 1$, allows us to replace $\{ \gamma^k \}_k$ \mbox{with a subsequence such that $\{ x_i^k \}_k \longrightarrow x_i^\star$ for all $i = 1, \ldots, N - 1$.}
	
	
	Consider now the sequences $\{ x_N^k \}_k$ and $\{ x_{N+1}^k \}_k$. By going over to a subsequence if necessary, we may assume w.l.o.g.~that $\{ \beta_k \}_k$ converges to $\beta^\star \in [0, \alpha_N]$ since the interval $[0, \alpha_N]$ is compact. If $\beta^\star \in (0, \alpha_N)$, then the argument from the previous paragraph applies equally to $\{ x_N^k \}_k$ and $\{ x_{N+1}^k \}_k$, that is, by iteratively going over to subsequences we may assume that $\{ x_i^k \}_k \longrightarrow x_i^\star$ also for $i = N, N + 1$. In this case, $\{ \gamma^k \}_k$ converges weakly to
	\begin{equation*}
		\gamma^\star
		\;\; = \;\;
		\sum_{i = 1}^{N - 1} \alpha_i \cdot \delta_{(x_i^\star, y_i)}
		\; + \;
		(\alpha_N - \beta^\star) \cdot \delta_{(x_N^\star, y_N)}
		\; + \;
		\beta^\star \cdot \delta_{(x_{N+1}^\star, y_N)}.
	\end{equation*}
	The measure $\gamma^\star$ resides in $\mathcal{D}_{N+1}$, and it satisfies the first integral constraint in~\eqref{opt:lift_2} by construction. It satisfies the second integral constraint since the metric $d$ is continuous in the topology that it generates. Moreover, $\gamma^\star$ attains the optimal value of~\eqref{opt:lift_2} since
	\begin{align}
		\int_{X \times Y} f(x_1) \, \mathrm{d}\gamma^\star (x_1, x_2)
		\;\; &= \;\;
		\sum_{i = 1}^{N - 1} \alpha_i \cdot f (x_i^\star) + (\alpha_N - \beta^\star) \cdot f (x_N^\star) + \beta^\star \cdot f (x_{N+1}^\star) \nonumber \\
		&\geq \;\;
		\mathop{\lim \sup}_{k \longrightarrow \infty} \left[
		\sum_{i = 1}^{N - 1} \alpha_i \cdot f (x_i^k)
		+ (\alpha_N - \beta_k) \cdot f (x_N^k)
		+ \beta_k \cdot f (x_{N+1}^k)
		\right] \nonumber \\
		&= \;\;
		\mathop{\lim \sup}_{k \longrightarrow \infty} \left[
		\int_{X \times Y} f(x_1) \, \mathrm{d}\gamma^k (x_1, x_2)
		\right]
		\;\; = \;\;
		\sup~\eqref{opt:lift_2}. \label{eq:opt}
	\end{align}
	Here, the first and the penultimate identity hold by the definitions of $\gamma^\star$ and $\gamma^k$, respectively, and the last identity follows from the fact that $\{ \gamma^k \}_k$ attains the optimal value of problem~\eqref{opt:lift_2} asymptotically. The inequality in~\eqref{eq:opt} follows from the upper semi-continuity of $f$ and the fact that $x_i^k \longrightarrow x_i^\star$, $i = 1, \ldots, N+1$, as well as $\beta_k \longrightarrow \beta^\star$.	
	
	Assume now that $\beta^\star = 0$; the case where $\beta^\star = \alpha_N$ is symmetric. A similar argument as before implies that by going over to a subsequence if necessary, we may assume w.l.o.g.~that $x_N^k \longrightarrow x_N^\star$. In this case, the sequence $\{ \gamma^k \}_k$ converges weakly to
	\begin{equation*}
		\gamma^\star
		\;\; = \;\;
		\sum_{i = 1}^N \alpha_i \cdot \delta_{(x_i^\star, y_i)},
	\end{equation*}
	even though $\{ x_{N+1}^k \}_k$ may not converge. Again, $\gamma^\star \in \mathcal{D}_{N+1}$, and $\gamma^\star$ satisfies the first integral constraint in problem~\eqref{opt:lift_2}. As for the second integral constraint in~\eqref{opt:lift_2}, we have
	\begin{align*}
		& \int_{X \times Y} d^p (x_1, x_2) \,\mathrm{d} \gamma^\star (x_1, x_2)
		\;\; = \;\;
		\sum_{i = 1}^N \alpha_i \cdot d^p (x_i^\star, y_i) \\
		\leq \;\;
		&\mathop{\lim \sup}_{k \longrightarrow \infty} \left[
		\sum_{i = 1}^{N - 1} \alpha_i \cdot d^p (x_i^k, y_i)
		+ (\alpha_N - \beta_k) \cdot d^p (x_N^k, y_N)
		+ \beta_k \cdot d^p (x_{N+1}^k, y_N)
		\right] \\
		= \;\;
		&\mathop{\lim \sup}_{k \longrightarrow \infty} \left[
		\int_{X \times Y} d^p (x_1, x_2) \,\mathrm{d} \gamma^k (x_1, x_2)
		\right]
		\;\; \leq \;\;
		r^p.
	\end{align*}
	Here, the first identity holds by definition of $\gamma^\star$. The inequality follows from the fact that the sum of $\sum_{i = 1}^{N - 1} \alpha_i \cdot d^p (x_i^k, y_i)$ and $(\alpha_N - \beta_k) \cdot d^p (x_N^k, y_N)$ converges to the expression $\sum_{i = 1}^N \alpha_i \cdot d^p (x_i^\star, y_i)$ on the left-hand side and the last expression $\beta_k \cdot d^p (x_{N+1}^k, y_N)$ is non-negative by construction. The last line, finally, follows from the definition of $\gamma^k$ and the feasibility of $\gamma^k$ in~\eqref{opt:lift_2}. We thus conclude that $\gamma^\star$ is feasible in problem~\eqref{opt:lift_2}.

	To see that $\gamma^\star$ attains the optimal value of problem~\eqref{opt:lift_2}, we observe that
	\begin{align}
		\int_{X \times Y} f(x_1) \, \mathrm{d}\gamma^\star (x_1, x_2)
		\;\; &= \;\;
		\sum_{i = 1}^N \alpha_i \cdot f (x_i^\star) \nonumber \\
		&\geq \;\;
		\mathop{\lim \sup}_{k \longrightarrow \infty} \left[
		\sum_{i = 1}^{N - 1} \alpha_i \cdot f (x_i^k)
		+ (\alpha_N - \beta_k) \cdot f (x_N^k)
		\right] \nonumber \\
		&\geq \;\;
		\mathop{\lim \sup}_{k \longrightarrow \infty} \left[
		\sum_{i = 1}^{N - 1} \alpha_i \cdot f (x_i^k)
		+ (\alpha_N - \beta_k) \cdot f (x_N^k)
		+ \beta_k \cdot f (x_{N+1}^k)
		\right] \nonumber \\
		&= \;\;
		\mathop{\lim \sup}_{k \longrightarrow \infty} \left[
		\int_{X \times Y} f(x_1) \, \mathrm{d}\gamma^k (x_1, x_2)
		\right]
		\;\; = \;\;
		\sup~\eqref{opt:lift_2}. \label{eq:aux_measure_opt}
	\end{align}
	Here, the identities can be justified as in~\eqref{eq:opt}. The first inequality holds since $f$ is upper semi-continuous, $x_i^k \longrightarrow x_i^\star$, $i = 1, \ldots, N$, and $\beta_k \longrightarrow 0$. We claim that
	\begin{equation}\label{eq:limit_vanishes}
		\mathop{\lim \sup}_{k \longrightarrow \infty} \left[ \beta_k \cdot f (x_{N+1}^k) \right] \leq 0,
	\end{equation}
	which proves the second inequality in~\eqref{eq:aux_measure_opt}. To show this, fix an arbitrary $\epsilon > 0$ and choose $R > 0$ large enough such that
	\begin{equation}\label{eq:daniels_magical_r_bound}
		\frac{r^p}{R^p} \cdot c \left[ 1 + \left( R + d (y_N, x_0) \right)^{p'} \right] \leq \epsilon,
	\end{equation}
	which exists since $p' < p$. Note that~\eqref{eq:daniels_magical_r_bound} is monotone in $R$, that is, if it is satisfied for some $R > 0$, it is also satisfied for all $R' > R$. Now select $k (\epsilon) \in \mathbb{N}$ large enough such that $\beta_k \cdot \sup_{x \in X} \{ f(x) \, : \, d (x, y_N) \leq R \} \leq \epsilon$ for all $k \geq k (\epsilon)$. Such a $k (\epsilon)$ exists because $f$ is upper semi-continuous and \mbox{hence its supremum over any compact set is finite. We claim that}
	\begin{equation*}
		\beta_k \cdot f (x_{N+1}^k) \leq \epsilon \qquad \forall k \geq k (\epsilon).
	\end{equation*}
	Fix any $k \geq k (\epsilon)$. If $d (x_{N + 1}^k, y_N) \leq R$, then the claim follows immediately from our choice of $k (\epsilon)$. If $d (x_{N + 1}^k, y_N) > R$, on the other hand, then the claim holds since
	\begin{align*}
		\beta_k \cdot f (x_{N+1}^k)
		\;\; &\leq \;\;
		\frac{r^p}{d^p (x_{N + 1}^k, y_N)} \cdot c \left[ 1 + d^{p'} (x_{N+1}^k, x_0) \right] \\
		&\leq \;\;
		\frac{r^p}{d^p (x_{N + 1}^k, y_N)} \cdot c \left[ 1 + \left( d (x_{N+1}^k, y_N) + d (y_N, x_0) \right)^{p'} \right] \\
		&\leq \;\;
		\frac{r^p}{R^p} \cdot c \left[ 1 + \left( R + d (y_N, x_0) \right)^{p'} \right]
		\;\; \leq \;\; \epsilon.
	\end{align*}
	Indeed, the first inequality holds since $\gamma^k$ satisfies the last integral constraint of problem~\eqref{opt:lift_2} and $f$ satisfies the growth condition of Theorem~\ref{thm:finite_supp}. The second inequality follows from the triangle inequality. The last two inequalities, finally, follow from~\eqref{eq:daniels_magical_r_bound} and the fact that $d (x_{N + 1}^k, y_N) > R$. We thus conclude that~\eqref{eq:limit_vanishes} indeed holds.
\end{proof}

\end{appendices}

\end{document}